\documentclass[12pt]{article}
\usepackage{amssymb,amsfonts,mathrsfs,amsmath,amsthm,graphicx,times,geometry,booktabs,needspace,float,tikz}
\usepackage[hidelinks]{hyperref}
\hypersetup{pdftitle={The least signless Laplacian eigenvalue of C3,C5-free graphs},pdfauthor={Qi Zhou}}
\newtheorem{Th}{\scshape Theorem}[section]
\newtheorem{Lem}[Th]{\scshape Lemma}

\theoremstyle{remark}

\numberwithin{equation}{section}
\newcommand{\qmin}{q_{\min}}
\newcommand{\RR}{\mathbb R}
\newcommand{\FF}{\mathbb F}
\newcommand{\ones}{\mathbf1}

\newcommand{\av}[1]{\left[\!\left[#1\right]\!\right]}

\begin{document}
\title{The least signless Laplacian eigenvalue of $\{C_3,C_5\}$-free graphs}
\author{Qi Zhou\\
\small\it School of Finance and Mathematics, Huainan Normal University, Huainan, China}
\date{}
\maketitle

\noindent{\bf Abstract:} Brandt [Discrete Math. 183 (1998) 17--25] conjectured that the least signless Laplacian eigenvalue of every regular triangle-free graph of order $n$ is at most $4n/25$. Using flag algebras, Balogh, Clemen, Lidick\'y, Norin and Volec [SIAM J. Discrete Math. 37 (2023) 1173--1179] established the stronger bound $15n/94$ for all triangle-free graphs. We investigate the effect of additionally excluding pentagons and prove that every $\{C_3,C_5\}$-free graph $G$ of order $n$ satisfies $\qmin(G)<0.0569n$, without any regularity assumption. The proof combines seven-vertex flag inequalities with local Rayleigh constraints that retain the least eigenvalue throughout the counting argument. An exact integer certificate establishes the required inequality.

\medskip
\noindent{\bf Keywords:} signless Laplacian eigenvalue; odd girth; flag algebras

\medskip
\noindent{\bf AMS Subject Classifications:} 05C50

\section{Introduction}

Let $G$ be a finite simple graph of order $n\ge1$, with adjacency matrix $A$ and diagonal degree matrix $D$. Its signless Laplacian is $Q=D+A$, and its least eigenvalue is denoted by $\qmin(G)$. We write $P_k$ and $C_k$ for the path and cycle on $k$ vertices, respectively. The odd girth of $G$ is the length of a shortest odd cycle. Thus $G$ is $\{C_3,C_5\}$-free if and only if its odd girth is at least seven. We study the largest possible value of $\qmin(G)/n$ in this paper.

The variational identity
\begin{equation}\label{eq:rayleigh}
\qmin(G)=\min_{0\ne x\in\RR^{V(G)}}
\frac{\sum_{uv\in E(G)}(x_u+x_v)^2}{\sum_{v\in V(G)}x_v^2}
\end{equation}
relates this eigenvalue to bipartiteness. In particular, $\qmin(G)=0$ if and only if $G$ has a bipartite connected component. If deleting $b$ edges makes $G$ bipartite, a $\{\pm1\}$-valued test vector in \eqref{eq:rayleigh} gives $\qmin(G)\le4b/n$. The connection between the smallest adjacency eigenvalue and bipartite subgraphs was studied by Alon and Sudakov \cite{alonsudakov}. For the signless Laplacian, bounds involving bipartite subgraphs and edge deletion were developed by Desai and Rao \cite{desai}, de Lima, Oliveira, de Abreu and Nikiforov \cite{lima2011}, and Fallat and Fan \cite{fallat}. Kirkland and Paul \cite{kirklandpaul} show how vectors with small signless Laplacian Rayleigh quotients can be used to identify bipartite subgraphs. The extremal problem under clique restrictions was studied by de Lima, Nikiforov and Oliveira \cite{lima2016}.

Brandt \cite{brandt} conjectured that $\qmin(G)\le4n/25$ for regular triangle-free graphs. Balogh, Clemen, Lidick\'y, Norin and Volec \cite{bclnv} proved the stronger inequality $\qmin(G)\le15n/94$ for every triangle-free graph. Their argument first bounds the least eigenvalue by a ratio involving strong homomorphisms of paths and then bounds that ratio using flag algebras. The additional exclusion of pentagons suggests a substantially smaller upper bound. Our main result is the following.

\begin{Th}\label{th:main}
Every $\{C_3,C_5\}$-free graph $G$ of order $n\ge1$ satisfies
\begin{equation}\label{eq:main}
\qmin(G)<\frac{569n}{10000}=0.0569n.
\end{equation}
\end{Th}

It also applies whenever further odd cycles are forbidden. After fixing five labelled vertices, we assign weights according to the neighbourhoods of the remaining vertices among these labels. Averaging Rayleigh inequalities gives a relation between two weighted subgraph counts. We establish the complementary counting inequality by expressing their difference as forty averaged flag squares and a strictly positive combination of seven-vertex densities. An integer certificate makes this coefficient comparison exact. This retains information about the least eigenvalue which would be lost by first replacing it with a single path ratio.

There is related work on the sum of the largest and smallest adjacency eigenvalues. Csikv\'ari \cite{csikvari} proves $\lambda_1(A)+\lambda_n(A)\le(3-2\sqrt2)n$ for all triangle-free graphs, extending an earlier bound of Brandt from the regular case. Abiad, Taranchuk and van Veluw \cite{atv} prove $\lambda_1(A)+\lambda_n(A)<0.0396n$ for odd girth at least seven, and Yip \cite{yip} obtains an upper bound of order $ng^{-3}\log^3g$ for odd girth at least $g$. For a regular graph, $\qmin(G)=\lambda_1(A)+\lambda_n(A)$, so the coefficient $0.0396$ is already stronger than ours on that subclass. Theorem~\ref{th:main} concerns general graphs, for which these two spectral quantities are not comparable. This distinction and the corresponding high-odd-girth question for the signless Laplacian are discussed by Abiad, Kumar and Pragada \cite{akp}.

The structure and local density of the same graph class have also been studied. Letzter and Snyder \cite{letztersnyder} proved that every $\{C_3,C_5\}$-free graph of order $n$ with minimum degree greater than $n/5$ is homomorphic to a generalised Andr\'asfai graph. Bedenknecht, Mota, Reiher and Schacht \cite{bmrs} established sparse-half bounds for graphs admitting such homomorphisms. In particular, under the same odd-girth and minimum-degree assumptions, there is a set of $\lfloor n/2\rfloor$ vertices spanning at most $n^2/98$ edges.

Odd-cycle counts provide another measure of the departure from bipartiteness. Grzesik \cite{grzesik} and, independently, Hatami, Hladk\'y, Kr\'al', Norine and Razborov \cite{hhknr} used flag algebras to prove that every triangle-free graph of order $n$ contains at most $(n/5)^5$ pentagons. Grzesik and Kielak \cite{grzesikkielak} subsequently proved by a combinatorial argument that, for every odd $k\ge7$, a graph of order $n$ and odd girth at least $k$ contains at most $(n/k)^k$ copies of $C_k$. In particular, their result applies to the same $\{C_3,C_5\}$-free class considered here. Our counting inequality uses weighted seven-vertex densities to control the least signless Laplacian eigenvalue.

We use the averaging principle of flag algebras \cite{razborov} to establish an explicit counting inequality. The proof presents the local test vectors first, then the flag squares, their expansion into seven-vertex densities, and the coordinatewise coefficient comparison. The finite calculation uses exact integers on the $89$ admissible seven-vertex graphs. Section~\ref{sec:proof} gives the proof, and Section~\ref{sec:conclusion} discusses the remaining extremal problem.

\section{Proof of Theorem~\ref{th:main}}\label{sec:proof}

The proof has two parts. We first associate two weighted subgraph counts, $\mathcal D$ and $\mathcal E$, with four local test vectors and bound $\qmin(G)\mathcal D(G)$ in terms of $\mathcal E(G)$. We then use flag algebras to establish a counting inequality with a positive remainder. Under the assumption $\qmin(G)\ge0.0569n$, these estimates are incompatible after taking a sufficiently large balanced blow-up. This follows the spectral-reduction--counting-inequality strategy of \cite{bclnv}; here the local counts retain information from five labelled vertices.

\subsection{A local spectral reduction}

Let $G$ have order $N\ge7$. Fix a five-vertex graph $\sigma$ with labels $0,\ldots,4$, and assign a real weight $w(S)$ to every possible neighbourhood $S\subseteq\{0,\ldots,4\}$ of an additional vertex. For an ordered tuple $R=(v_0,\ldots,v_4)$ of distinct vertices inducing precisely $\sigma$, put
\[
w_R(x)=w(\{i:xv_i\in E(G)\})\qquad(x\notin R).
\]
Thus the weight of $x$ depends only on its adjacency to the five roots. If $G$ is $\{C_3,C_5\}$-free, only neighbourhoods giving $\{C_3,C_5\}$-free extensions of $\sigma$ arise.

To apply \eqref{eq:rayleigh}, consider the local test vector $z\in\RR^{V(G)}$ defined by
\[
z_x=\begin{cases}0,&x\in R,\\ w_R(x),&x\notin R.\end{cases}
\]
Write $d_R(x)=|N_G(x)\cap R|$. Splitting the edges according to whether they meet $R$ gives
\begin{align*}
z^Tz&=\sum_{x\notin R}w_R(x)^2,\\
z^TQz&=\sum_{\substack{xy\in E(G)\\x,y\notin R}}(w_R(x)+w_R(y))^2
+\sum_{x\notin R}d_R(x)w_R(x)^2.
\end{align*}
Indeed, an edge inside $R$ contributes zero, while an edge from $x\notin R$ to a root contributes $w_R(x)^2$. The first sum in the quadratic form depends on two vertices outside $R$, whereas the squared norm depends on one. We therefore introduce the normalised quantities
\begin{align}
D_R(w)&=\frac1{N-5}\sum_{x\notin R}w_R(x)^2,\label{eq:rootD}\\
E_R(w)&=\frac1{(N-5)(N-6)}
\sum_{\substack{xy\in E(G)\\x,y\notin R}}(w_R(x)+w_R(y))^2.\label{eq:rootE}
\end{align}
The denominator in \eqref{eq:rootD} counts vertices outside $R$, and that in \eqref{eq:rootE} counts ordered pairs of distinct vertices outside $R$. Each edge in \eqref{eq:rootE} is counted only once. Consequently, $E_R(w)$ is one half of the average of $\ones_{xy\in E(G)}(w_R(x)+w_R(y))^2$ over those ordered pairs. This convention matches the flag products used in the next subsection.

Let $D_\sigma(w;G)$ and $E_\sigma(w;G)$ be the averages of these quantities over all ordered five-tuples of distinct vertices, with contribution zero from a tuple of the wrong type. We always use this unconditional averaging convention. The terms involving roots have $d_R(x)\le5$, which yields the following estimate.

\begin{Lem}\label{lem:local}
For every choice of $\sigma,w$ and every graph $G$ of order $N\ge7$,
\begin{equation}\label{eq:localrayleigh}
\qmin(G)D_\sigma(w;G)
\le (N-6)E_\sigma(w;G)+5D_\sigma(w;G).
\end{equation}
\end{Lem}

\begin{proof}
For each tuple $R$ inducing $\sigma$, apply $\qmin(G)z^Tz\le z^TQz$ to the vector above; this inequality also holds when $z=0$. Substituting \eqref{eq:rootD} and \eqref{eq:rootE} and dividing by $N-5$ gives
\begin{align*}
\qmin(G)D_R(w)
&\le (N-6)E_R(w)+\frac1{N-5}\sum_{x\notin R}d_R(x)w_R(x)^2\\
&\le (N-6)E_R(w)+5D_R(w).
\end{align*}
Averaging over all ordered five-tuples, with zero contribution from tuples of the wrong type, proves \eqref{eq:localrayleigh}.
\end{proof}

We shall use four choices of $(\sigma,w)$: two with $\sigma$ edgeless, one with $\sigma=K_{1,4}$, and one with $\sigma=P_5$. Their precise labels appear as $\sigma_1,\sigma_2,\sigma_3$ in Figure~\ref{fig:types}; their weights are specified by the four columns of the matrices $V_\sigma$ in Section~\ref{subsec:certificate}. Define $\mathcal D(G)$ and $\mathcal E(G)$ to be the respective sums of $D_\sigma(w;G)$ and $E_\sigma(w;G)$ over these four choices. Lemma~\ref{lem:local} gives
\begin{equation}\label{eq:spectralreduction}
\qmin(G)\mathcal D(G)\le (N-6)\mathcal E(G)+5\mathcal D(G).
\end{equation}
Both counts are nonnegative, and $\mathcal D(G)$ is bounded by a constant depending only on the four fixed weight vectors.

\subsection{Flags, products and nonnegative squares}

Call a graph {\it admissible} if it is $\{C_3,C_5\}$-free. Let $\mathcal H_7$ be the isomorphism classes of admissible seven-vertex graphs, and let $p(H,G)$ be the proportion of seven-element subsets of $V(G)$ that induce $H$. There are $89$ members of $\mathcal H_7$: the $88$ bipartite graphs and $C_7$. To see the latter assertion, an admissible nonbipartite graph on seven vertices must contain a spanning $7$-cycle; every chord would create a triangle or a pentagon.

A type $\tau$ is an admissible graph with fixed, distinct vertex labels. A $\tau$-flag is an admissible graph containing these labelled vertices and inducing exactly $\tau$ on them. Isomorphisms of flags fix every label. Figure~\ref{fig:types} shows the types used in the proof: $\tau_1,\ldots,\tau_7$ support the ordinary flag squares, and $\sigma_1,\sigma_2,\sigma_3$ support the local Rayleigh expressions. Throughout the figures, hollow vertices carry fixed labels and solid vertices are unlabelled. Isolated labelled vertices are retained.

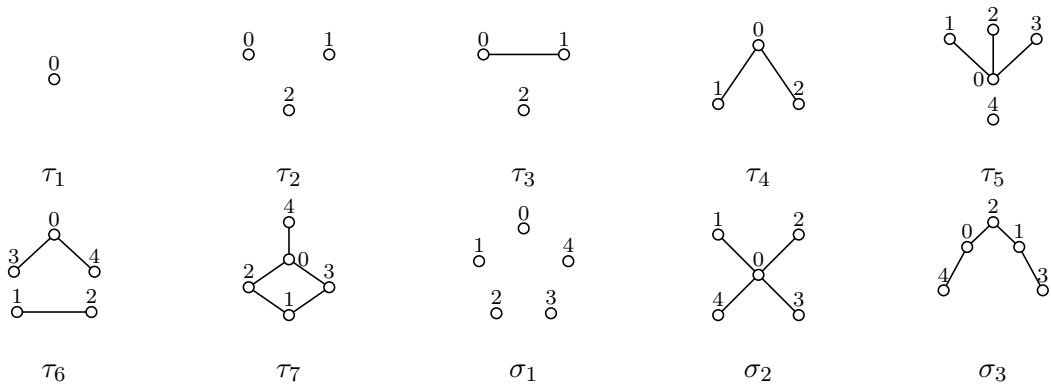
\begin{figure}[!htb]
\centering
\begin{minipage}[b]{3.0cm}\centering
\begin{tikzpicture}[x=0.82cm,y=0.82cm,line width=0.6pt]
\path[use as bounding box] (-1.02,-1.04) rectangle (1.02,1.14);
\coordinate (v0) at (0,0);
\draw[fill=white] (v0) circle[radius=2pt];
\node[above=3pt,inner sep=0pt,font=\scriptsize] at (v0) {$ 0 $};
\end{tikzpicture}\\[-0.6ex]
{\small $\tau_1$}
\end{minipage}
\begin{minipage}[b]{3.0cm}\centering
\begin{tikzpicture}[x=0.82cm,y=0.82cm,line width=0.6pt]
\path[use as bounding box] (-1.02,-1.04) rectangle (1.02,1.14);
\coordinate (v0) at (-0.65,0.4);
\coordinate (v1) at (0.65,0.4);
\coordinate (v2) at (0,-0.5);
\draw[fill=white] (v0) circle[radius=2pt];
\node[above=3pt,inner sep=0pt,font=\scriptsize] at (v0) {$ 0 $};
\draw[fill=white] (v1) circle[radius=2pt];
\node[above=3pt,inner sep=0pt,font=\scriptsize] at (v1) {$ 1 $};
\draw[fill=white] (v2) circle[radius=2pt];
\node[above=3pt,inner sep=0pt,font=\scriptsize] at (v2) {$ 2 $};
\end{tikzpicture}\\[-0.6ex]
{\small $\tau_2$}
\end{minipage}
\begin{minipage}[b]{3.0cm}\centering
\begin{tikzpicture}[x=0.82cm,y=0.82cm,line width=0.6pt]
\path[use as bounding box] (-1.02,-1.04) rectangle (1.02,1.14);
\coordinate (v0) at (-0.65,0.4);
\coordinate (v1) at (0.65,0.4);
\coordinate (v2) at (0,-0.5);
\draw (v0)--(v1);
\draw[fill=white] (v0) circle[radius=2pt];
\node[above=3pt,inner sep=0pt,font=\scriptsize] at (v0) {$ 0 $};
\draw[fill=white] (v1) circle[radius=2pt];
\node[above=3pt,inner sep=0pt,font=\scriptsize] at (v1) {$ 1 $};
\draw[fill=white] (v2) circle[radius=2pt];
\node[above=3pt,inner sep=0pt,font=\scriptsize] at (v2) {$ 2 $};
\end{tikzpicture}\\[-0.6ex]
{\small $\tau_3$}
\end{minipage}
\begin{minipage}[b]{3.0cm}\centering
\begin{tikzpicture}[x=0.82cm,y=0.82cm,line width=0.6pt]
\path[use as bounding box] (-1.02,-1.04) rectangle (1.02,1.14);
\coordinate (v0) at (0,0.55);
\coordinate (v1) at (-0.65,-0.4);
\coordinate (v2) at (0.65,-0.4);
\draw (v0)--(v1);
\draw (v0)--(v2);
\draw[fill=white] (v0) circle[radius=2pt];
\node[above=3pt,inner sep=0pt,font=\scriptsize] at (v0) {$ 0 $};
\draw[fill=white] (v1) circle[radius=2pt];
\node[above=3pt,inner sep=0pt,font=\scriptsize] at (v1) {$ 1 $};
\draw[fill=white] (v2) circle[radius=2pt];
\node[above=3pt,inner sep=0pt,font=\scriptsize] at (v2) {$ 2 $};
\end{tikzpicture}\\[-0.6ex]
{\small $\tau_4$}
\end{minipage}
\begin{minipage}[b]{3.0cm}\centering
\begin{tikzpicture}[x=0.82cm,y=0.82cm,line width=0.6pt]
\path[use as bounding box] (-1.02,-1.04) rectangle (1.02,1.14);
\coordinate (v0) at (0,0);
\coordinate (v1) at (-0.7,0.65);
\coordinate (v2) at (0,0.8);
\coordinate (v3) at (0.7,0.65);
\coordinate (v4) at (0,-0.65);
\draw (v0)--(v1);
\draw (v0)--(v2);
\draw (v0)--(v3);
\draw[fill=white] (v0) circle[radius=2pt];
\node[left=3pt,inner sep=0pt,font=\scriptsize] at (v0) {$ 0 $};
\draw[fill=white] (v1) circle[radius=2pt];
\node[above=3pt,inner sep=0pt,font=\scriptsize] at (v1) {$ 1 $};
\draw[fill=white] (v2) circle[radius=2pt];
\node[above=3pt,inner sep=0pt,font=\scriptsize] at (v2) {$ 2 $};
\draw[fill=white] (v3) circle[radius=2pt];
\node[above=3pt,inner sep=0pt,font=\scriptsize] at (v3) {$ 3 $};
\draw[fill=white] (v4) circle[radius=2pt];
\node[above=3pt,inner sep=0pt,font=\scriptsize] at (v4) {$ 4 $};
\end{tikzpicture}\\[-0.6ex]
{\small $\tau_5$}
\end{minipage}\par\medskip\begin{minipage}[b]{3.0cm}\centering
\begin{tikzpicture}[x=0.82cm,y=0.82cm,line width=0.6pt]
\path[use as bounding box] (-1.02,-1.04) rectangle (1.02,1.14);
\coordinate (v0) at (0,0.65);
\coordinate (v1) at (-0.6,-0.6);
\coordinate (v2) at (0.6,-0.6);
\coordinate (v3) at (-0.65,0.05);
\coordinate (v4) at (0.65,0.05);
\draw (v0)--(v3);
\draw (v0)--(v4);
\draw (v1)--(v2);
\draw[fill=white] (v0) circle[radius=2pt];
\node[above=3pt,inner sep=0pt,font=\scriptsize] at (v0) {$ 0 $};
\draw[fill=white] (v1) circle[radius=2pt];
\node[above=3pt,inner sep=0pt,font=\scriptsize] at (v1) {$ 1 $};
\draw[fill=white] (v2) circle[radius=2pt];
\node[above=3pt,inner sep=0pt,font=\scriptsize] at (v2) {$ 2 $};
\draw[fill=white] (v3) circle[radius=2pt];
\node[above=3pt,inner sep=0pt,font=\scriptsize] at (v3) {$ 3 $};
\draw[fill=white] (v4) circle[radius=2pt];
\node[above=3pt,inner sep=0pt,font=\scriptsize] at (v4) {$ 4 $};
\end{tikzpicture}\\[-0.6ex]
{\small $\tau_6$}
\end{minipage}
\begin{minipage}[b]{3.0cm}\centering
\begin{tikzpicture}[x=0.82cm,y=0.82cm,line width=0.6pt]
\path[use as bounding box] (-1.02,-1.04) rectangle (1.02,1.14);
\coordinate (v0) at (0,0.25);
\coordinate (v1) at (0,-0.65);
\coordinate (v2) at (-0.65,-0.2);
\coordinate (v3) at (0.65,-0.2);
\coordinate (v4) at (0,0.85);
\draw (v0)--(v2);
\draw (v0)--(v3);
\draw (v0)--(v4);
\draw (v1)--(v2);
\draw (v1)--(v3);
\draw[fill=white] (v0) circle[radius=2pt];
\node[right=3pt,inner sep=0pt,font=\scriptsize] at (v0) {$ 0 $};
\draw[fill=white] (v1) circle[radius=2pt];
\node[above=3pt,inner sep=0pt,font=\scriptsize] at (v1) {$ 1 $};
\draw[fill=white] (v2) circle[radius=2pt];
\node[above=3pt,inner sep=0pt,font=\scriptsize] at (v2) {$ 2 $};
\draw[fill=white] (v3) circle[radius=2pt];
\node[above=3pt,inner sep=0pt,font=\scriptsize] at (v3) {$ 3 $};
\draw[fill=white] (v4) circle[radius=2pt];
\node[above=3pt,inner sep=0pt,font=\scriptsize] at (v4) {$ 4 $};
\end{tikzpicture}\\[-0.6ex]
{\small $\tau_7$}
\end{minipage}
\begin{minipage}[b]{3.0cm}\centering
\begin{tikzpicture}[x=0.82cm,y=0.82cm,line width=0.6pt]
\path[use as bounding box] (-1.02,-1.04) rectangle (1.02,1.14);
\coordinate (v0) at (0,0.75);
\coordinate (v1) at (-0.72,0.22);
\coordinate (v2) at (-0.44,-0.62);
\coordinate (v3) at (0.44,-0.62);
\coordinate (v4) at (0.72,0.22);
\draw[fill=white] (v0) circle[radius=2pt];
\node[above=3pt,inner sep=0pt,font=\scriptsize] at (v0) {$ 0 $};
\draw[fill=white] (v1) circle[radius=2pt];
\node[above=3pt,inner sep=0pt,font=\scriptsize] at (v1) {$ 1 $};
\draw[fill=white] (v2) circle[radius=2pt];
\node[above=3pt,inner sep=0pt,font=\scriptsize] at (v2) {$ 2 $};
\draw[fill=white] (v3) circle[radius=2pt];
\node[above=3pt,inner sep=0pt,font=\scriptsize] at (v3) {$ 3 $};
\draw[fill=white] (v4) circle[radius=2pt];
\node[above=3pt,inner sep=0pt,font=\scriptsize] at (v4) {$ 4 $};
\end{tikzpicture}\\[-0.6ex]
{\small $\sigma_1$}
\end{minipage}
\begin{minipage}[b]{3.0cm}\centering
\begin{tikzpicture}[x=0.82cm,y=0.82cm,line width=0.6pt]
\path[use as bounding box] (-1.02,-1.04) rectangle (1.02,1.14);
\coordinate (v0) at (0,0);
\coordinate (v1) at (-0.65,0.65);
\coordinate (v2) at (0.65,0.65);
\coordinate (v3) at (0.65,-0.65);
\coordinate (v4) at (-0.65,-0.65);
\draw (v0)--(v1);
\draw (v0)--(v2);
\draw (v0)--(v3);
\draw (v0)--(v4);
\draw[fill=white] (v0) circle[radius=2pt];
\node[above=3pt,inner sep=0pt,font=\scriptsize] at (v0) {$ 0 $};
\draw[fill=white] (v1) circle[radius=2pt];
\node[above=3pt,inner sep=0pt,font=\scriptsize] at (v1) {$ 1 $};
\draw[fill=white] (v2) circle[radius=2pt];
\node[above=3pt,inner sep=0pt,font=\scriptsize] at (v2) {$ 2 $};
\draw[fill=white] (v3) circle[radius=2pt];
\node[above=3pt,inner sep=0pt,font=\scriptsize] at (v3) {$ 3 $};
\draw[fill=white] (v4) circle[radius=2pt];
\node[above=3pt,inner sep=0pt,font=\scriptsize] at (v4) {$ 4 $};
\end{tikzpicture}\\[-0.6ex]
{\small $\sigma_2$}
\end{minipage}
\begin{minipage}[b]{3.0cm}\centering
\begin{tikzpicture}[x=0.82cm,y=0.82cm,line width=0.6pt]
\path[use as bounding box] (-1.02,-1.04) rectangle (1.02,1.14);
\coordinate (v0) at (-0.42,0.45);
\coordinate (v1) at (0.42,0.45);
\coordinate (v2) at (0,0.85);
\coordinate (v3) at (0.8,-0.25);
\coordinate (v4) at (-0.8,-0.25);
\draw (v0)--(v2);
\draw (v0)--(v4);
\draw (v1)--(v2);
\draw (v1)--(v3);
\draw[fill=white] (v0) circle[radius=2pt];
\node[above=3pt,inner sep=0pt,font=\scriptsize] at (v0) {$ 0 $};
\draw[fill=white] (v1) circle[radius=2pt];
\node[above=3pt,inner sep=0pt,font=\scriptsize] at (v1) {$ 1 $};
\draw[fill=white] (v2) circle[radius=2pt];
\node[above=3pt,inner sep=0pt,font=\scriptsize] at (v2) {$ 2 $};
\draw[fill=white] (v3) circle[radius=2pt];
\node[above=3pt,inner sep=0pt,font=\scriptsize] at (v3) {$ 3 $};
\draw[fill=white] (v4) circle[radius=2pt];
\node[above=3pt,inner sep=0pt,font=\scriptsize] at (v4) {$ 4 $};
\end{tikzpicture}\\[-0.6ex]
{\small $\sigma_3$}
\end{minipage}

\caption{The labelled types.}\label{fig:types}
\end{figure}

\Needspace{9\baselineskip}
For a type $\tau$ of order $k\in\{1,3,5\}$, write $F_1^\tau,\ldots,F_{d_\tau}^\tau$ for all its flags of order $\ell=(7+k)/2$. Choose an ordered tuple $(v_0,\ldots,v_6)$ of seven distinct vertices of $G$ uniformly at random. The first flag uses $v_0,\ldots,v_{\ell-1}$, and the second uses $v_0,\ldots,v_{k-1},v_\ell,\ldots,v_6$. Tuples of the wrong root type contribute zero. Define
\[
\av{F_i^\tau F_j^\tau}_G
=\Pr(\text{the first and second flags are }F_i^\tau,F_j^\tau).
\]
Extend this notation bilinearly: for $L=\sum_i a_iF_i^\tau$, put
\begin{equation}\label{eq:flagaverage}
\av{L^2}_G=\sum_{i,j}a_i a_j\av{F_i^\tau F_j^\tau}_G.
\end{equation}
This is the usual averaged flag square. With sampling without replacement it need not be nonnegative on a fixed finite graph; the following asymptotic form is sufficient.

\begin{Lem}\label{lem:flagsquare}
For every fixed linear combination $L$ of flags of one type, there is a constant $C_L$ such that
\begin{equation}\label{eq:flagsquarenonneg}
\av{L^2}_G\ge-\frac{C_L}{N}
\end{equation}
for every admissible graph $G$ of sufficiently large order $N$.
\end{Lem}

\begin{proof}
Fix roots inducing $\tau$, and let $f_i$ be the probability of obtaining $F_i^\tau$ by choosing a uniformly random $(\ell-k)$-set of other vertices. Independent choices for two such sets $S,T$ give $(\sum_i a_i f_i)^2\ge0$. Moreover,
\[
\Pr(S\cap T\ne\varnothing)\le\frac{(\ell-k)^2}{N-k}=O(N^{-1}),
\]
uniformly in the roots. Conditional on $S\cap T=\varnothing$, the distribution is exactly that of the two flags defined above: every ordered pair of disjoint extension sets has the same number of internal orderings. Since the coefficients are fixed, this conditioning changes the average by $O(N^{-1})$. Averaging over the roots, with zero contribution from an incorrect type, proves the assertion.
\end{proof}

To make the averaging concrete, consider $F_8^{\sigma_3}$ and $F_6^{\sigma_3}$ in Figure~\ref{fig:product}. The root type is the path $4,0,2,1,3$. The extension of $F_8^{\sigma_3}$ is adjacent only to root $4$, and that of $F_6^{\sigma_3}$ only to root $3$. Write $x$ for the extension on the left and $y$ for the one on the right. Their mutual edge is the only unspecified pair: its absence gives $P_7$, and its presence gives $C_7$. Consequently,
\begin{align}
\av{F_8^{\sigma_3}F_6^{\sigma_3}}_G
&=\frac1{2520}p(P_7,G)+\frac1{360}p(C_7,G),\label{eq:productexample}\\
\av{F_8^{\sigma_3}F_6^{\sigma_3}\,\mathbf1_{xy\in E}}_G
&=\frac1{360}p(C_7,G).\label{eq:edgeexample}
\end{align}
Indeed, among the $7!=5040$ orderings, exactly two orderings of $P_7$ and fourteen of $C_7$ produce this ordered flag pair. The edge condition keeps only the latter. Here and below $x,y$ denote the two extension vertices of a five-root flag product.

\begin{figure}[H]
\centering
\begin{minipage}[b]{3.75cm}\centering
\begin{tikzpicture}[x=0.82cm,y=0.82cm,line width=0.6pt]
\path[use as bounding box] (-1.2,-1.03) rectangle (1.2,1.25);
\coordinate (v0) at (-0.58,0.65);
\coordinate (v1) at (0.58,0.65);
\coordinate (v2) at (0,1);
\coordinate (v3) at (1,0);
\coordinate (v4) at (-1,0);
\coordinate (v5) at (-0.62,-0.85);
\draw (v0)--(v2);
\draw (v0)--(v4);
\draw (v1)--(v2);
\draw (v1)--(v3);
\draw (v4)--(v5);
\draw[fill=white] (v0) circle[radius=2pt];
\node[above=3pt,inner sep=0pt,font=\scriptsize] at (v0) {$ 0 $};
\draw[fill=white] (v1) circle[radius=2pt];
\node[above=3pt,inner sep=0pt,font=\scriptsize] at (v1) {$ 1 $};
\draw[fill=white] (v2) circle[radius=2pt];
\node[above=3pt,inner sep=0pt,font=\scriptsize] at (v2) {$ 2 $};
\draw[fill=white] (v3) circle[radius=2pt];
\node[above=3pt,inner sep=0pt,font=\scriptsize] at (v3) {$ 3 $};
\draw[fill=white] (v4) circle[radius=2pt];
\node[above=3pt,inner sep=0pt,font=\scriptsize] at (v4) {$ 4 $};
\draw[fill=black] (v5) circle[radius=2pt];
\end{tikzpicture}\\[-0.6ex]
{\small $F^{\sigma_3}_8$}
\end{minipage}
\begin{minipage}[b]{3.75cm}\centering
\begin{tikzpicture}[x=0.82cm,y=0.82cm,line width=0.6pt]
\path[use as bounding box] (-1.2,-1.03) rectangle (1.2,1.25);
\coordinate (v0) at (-0.58,0.65);
\coordinate (v1) at (0.58,0.65);
\coordinate (v2) at (0,1);
\coordinate (v3) at (1,0);
\coordinate (v4) at (-1,0);
\coordinate (v5) at (0.62,-0.85);
\draw (v0)--(v2);
\draw (v0)--(v4);
\draw (v1)--(v2);
\draw (v1)--(v3);
\draw (v3)--(v5);
\draw[fill=white] (v0) circle[radius=2pt];
\node[above=3pt,inner sep=0pt,font=\scriptsize] at (v0) {$ 0 $};
\draw[fill=white] (v1) circle[radius=2pt];
\node[above=3pt,inner sep=0pt,font=\scriptsize] at (v1) {$ 1 $};
\draw[fill=white] (v2) circle[radius=2pt];
\node[above=3pt,inner sep=0pt,font=\scriptsize] at (v2) {$ 2 $};
\draw[fill=white] (v3) circle[radius=2pt];
\node[above=3pt,inner sep=0pt,font=\scriptsize] at (v3) {$ 3 $};
\draw[fill=white] (v4) circle[radius=2pt];
\node[above=3pt,inner sep=0pt,font=\scriptsize] at (v4) {$ 4 $};
\draw[fill=black] (v5) circle[radius=2pt];
\end{tikzpicture}\\[-0.6ex]
{\small $F^{\sigma_3}_6$}
\end{minipage}
\begin{minipage}[b]{3.75cm}\centering
\begin{tikzpicture}[x=0.82cm,y=0.82cm,line width=0.6pt]
\path[use as bounding box] (-1.2,-1.03) rectangle (1.2,1.25);
\coordinate (v0) at (-0.58,0.65);
\coordinate (v1) at (0.58,0.65);
\coordinate (v2) at (0,1);
\coordinate (v3) at (1,0);
\coordinate (v4) at (-1,0);
\coordinate (v5) at (-0.62,-0.85);
\coordinate (v6) at (0.62,-0.85);
\draw (v0)--(v2);
\draw (v0)--(v4);
\draw (v1)--(v2);
\draw (v1)--(v3);
\draw (v4)--(v5);
\draw (v3)--(v6);
\draw[fill=white] (v0) circle[radius=2pt];
\node[above=3pt,inner sep=0pt,font=\scriptsize] at (v0) {$ 0 $};
\draw[fill=white] (v1) circle[radius=2pt];
\node[above=3pt,inner sep=0pt,font=\scriptsize] at (v1) {$ 1 $};
\draw[fill=white] (v2) circle[radius=2pt];
\node[above=3pt,inner sep=0pt,font=\scriptsize] at (v2) {$ 2 $};
\draw[fill=white] (v3) circle[radius=2pt];
\node[above=3pt,inner sep=0pt,font=\scriptsize] at (v3) {$ 3 $};
\draw[fill=white] (v4) circle[radius=2pt];
\node[above=3pt,inner sep=0pt,font=\scriptsize] at (v4) {$ 4 $};
\draw[fill=black] (v5) circle[radius=2pt];
\draw[fill=black] (v6) circle[radius=2pt];
\end{tikzpicture}\\[-0.6ex]
{\small $P_7\quad(xy\notin E)$}
\end{minipage}
\begin{minipage}[b]{3.75cm}\centering
\begin{tikzpicture}[x=0.82cm,y=0.82cm,line width=0.6pt]
\path[use as bounding box] (-1.2,-1.03) rectangle (1.2,1.25);
\coordinate (v0) at (-0.58,0.65);
\coordinate (v1) at (0.58,0.65);
\coordinate (v2) at (0,1);
\coordinate (v3) at (1,0);
\coordinate (v4) at (-1,0);
\coordinate (v5) at (-0.62,-0.85);
\coordinate (v6) at (0.62,-0.85);
\draw (v0)--(v2);
\draw (v0)--(v4);
\draw (v1)--(v2);
\draw (v1)--(v3);
\draw (v4)--(v5);
\draw (v3)--(v6);
\draw (v5)--(v6);
\draw[fill=white] (v0) circle[radius=2pt];
\node[above=3pt,inner sep=0pt,font=\scriptsize] at (v0) {$ 0 $};
\draw[fill=white] (v1) circle[radius=2pt];
\node[above=3pt,inner sep=0pt,font=\scriptsize] at (v1) {$ 1 $};
\draw[fill=white] (v2) circle[radius=2pt];
\node[above=3pt,inner sep=0pt,font=\scriptsize] at (v2) {$ 2 $};
\draw[fill=white] (v3) circle[radius=2pt];
\node[above=3pt,inner sep=0pt,font=\scriptsize] at (v3) {$ 3 $};
\draw[fill=white] (v4) circle[radius=2pt];
\node[above=3pt,inner sep=0pt,font=\scriptsize] at (v4) {$ 4 $};
\draw[fill=black] (v5) circle[radius=2pt];
\draw[fill=black] (v6) circle[radius=2pt];
\end{tikzpicture}\\[-0.6ex]
{\small $C_7\quad(xy\in E)$}
\end{minipage}
\caption{A flag product over $\sigma_3$.}\label{fig:product}
\end{figure}
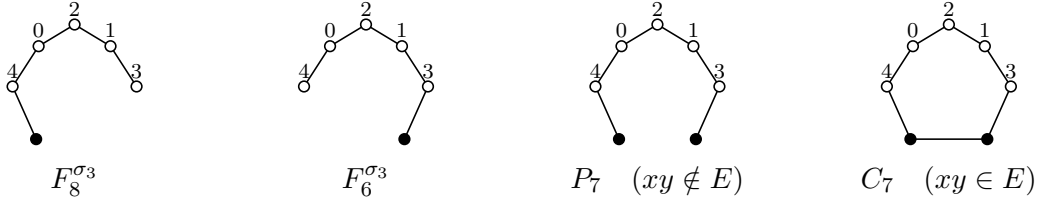

We now express the spectral counts using the same flag products. Fix roots $R$ inducing a five-vertex type $\sigma$, and choose distinct vertices $x,y$ outside $R$ uniformly in order. Their marginal distributions are identical, so
\[
D_R(w)=\mathbb E[w_R(x)^2]
=\frac12\mathbb E[w_R(x)^2+w_R(y)^2].
\]
Let $w_i$ be the weight assigned to the neighbourhood described by $F_i^\sigma$. Classifying the ordered pair $(x,y)$ by its two flags, and then averaging over the roots, gives
\begin{align}
D_\sigma(w;G)
&=\frac12\sum_{i,j}(w_i^2+w_j^2)\av{F_i^\sigma F_j^\sigma}_G,\label{eq:flagD}\\
E_\sigma(w;G)
&=\frac12\sum_{i,j}(w_i+w_j)^2
\av{F_i^\sigma F_j^\sigma\,\mathbf1_{xy\in E}}_G.\label{eq:flagE}
\end{align}
In \eqref{eq:flagD}, the factor $1/2$ comes from the equality of the two marginal distributions. In \eqref{eq:flagE}, it accounts for the two orders of each edge, as in \eqref{eq:rootE}. Five roots and two extension vertices give seven vertices in total, so both $\mathcal D$ and $\mathcal E$ are linear combinations of seven-vertex induced densities. The counting inequality needed for \eqref{eq:spectralreduction} can therefore be sought in this common space.

\subsection{The flag algebra inequality}\label{subsec:certificate}

Our aim is to express $r\mathcal D-\mathcal E$ as a sum of averaged flag squares and a combination of seven-vertex densities with strictly positive coefficients. For a fixed $r$, allowing positive semidefinite Gram matrices for the ordinary squares and the local Rayleigh expressions turns the coefficient conditions into a semidefinite feasibility problem. Numerical optimisation supplies candidate matrices; rational factors and exact coefficient comparisons establish the certificate used below. The displayed types and columns are those retained in this certificate.

Put $r_0=569/10000$ and $s=10^8$. The certificate specifies integer matrices $U_\tau$ for $\tau_1,\ldots,\tau_7$, and $V_\sigma$ for $\sigma_1,\sigma_2,\sigma_3$. Each row corresponds to one flag. A column $u$ of $U_\tau$ defines the linear combination
\begin{equation}\label{eq:columnflag}
L_u=\frac1s\sum_{i=1}^{d_\tau}u_iF_i^\tau,
\end{equation}
and a column $c$ of $V_\sigma$ defines the local weights $w_i=c_i/s$ in \eqref{eq:flagD}--\eqref{eq:flagE}. Table~\ref{tab:types} lists the type codes and flag dimensions. Its last two columns count ordinary flag squares and local test vectors, respectively, giving $40$ squares and four test vectors in total.

\begin{table}[H]
\centering
\caption{Types and flag dimensions.}\label{tab:types}
\begin{tabular}{crrrrr}
\toprule
Type & Order & Code & Number of flags & $U$ columns & $V$ columns\\
\midrule
$\tau_1$ & 1 & 0 & 12 & 5 & 0\\
$\tau_2$ & 3 & 0 & 50 & 7 & 0\\
$\tau_3$ & 3 & 1 & 31 & 14 & 0\\
$\tau_4$ & 3 & 3 & 23 & 11 & 0\\
$\tau_5$ & 5 & 7 & 18 & 1 & 0\\
$\tau_6$ & 5 & 28 & 15 & 1 & 0\\
$\tau_7$ & 5 & 62 & 11 & 1 & 0\\
$\sigma_1$ & 5 & 0 & 32 & 0 & 2\\
$\sigma_2$ & 5 & 15 & 17 & 0 & 1\\
$\sigma_3$ & 5 & 58 & 11 & 0 & 1\\
\bottomrule
\end{tabular}
\end{table}

For reproducibility, encode the edges of a graph on $0,\ldots,m-1$ by $\sum_{\{i,j\}\in E}2^{\operatorname{index}(i,j)}$, with pairs $i<j$ listed lexicographically from index zero. Type codes are the smallest over all relabellings. Flag codes are the smallest over permutations of unlabelled vertices only, and the flags are ordered by increasing code, with indices starting at $1$. The complete integer matrices are specified in the single ancillary file described at the end of the paper and form part of the certificate. In that file, the blocks named \texttt{gram\_(k, code)} are $U_\tau$, and those named \texttt{spectral\_(5, code)} are $V_\sigma$, where \texttt{code} is the type code in Table~\ref{tab:types}.

One of the $40$ squares has a particularly short form. The second column of $U_{\tau_2}$ gives exactly
\begin{equation}\label{eq:examplesquare}
\frac{202179961}{10^{16}}
\av{\left(F_7^{\tau_2}-F_9^{\tau_2}-F_{14}^{\tau_2}
+F_{23}^{\tau_2}+F_{28}^{\tau_2}-F_{30}^{\tau_2}\right)^2}_G.
\end{equation}
The six flags are shown in Figure~\ref{fig:square}. Their codes are, respectively, $44,100,140,224,388,416$. Although their underlying unlabelled graphs are all $P_4\cup K_1$, they are distinct flags because the root labels are fixed. The common factor is $14219^2/s^2$. This illustrates how a column of the certificate yields an explicit nonnegative flag expression, up to the finite sampling error in Lemma~\ref{lem:flagsquare}.

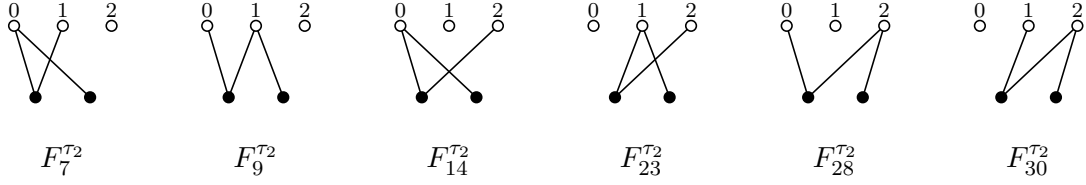
\begin{figure}[H]
\centering
\begin{minipage}[b]{2.45cm}\centering
\begin{tikzpicture}[x=0.86cm,y=0.86cm,line width=0.6pt]
\path[use as bounding box] (-1.02,-1.04) rectangle (1.02,1.14);
\coordinate (v0) at (-0.75,0.55);
\coordinate (v1) at (0,0.55);
\coordinate (v2) at (0.75,0.55);
\coordinate (v3) at (-0.42,-0.55);
\coordinate (v4) at (0.42,-0.55);
\draw (v0)--(v3);
\draw (v0)--(v4);
\draw (v1)--(v3);
\draw[fill=white] (v0) circle[radius=2pt];
\node[above=3pt,inner sep=0pt,font=\scriptsize] at (v0) {$ 0 $};
\draw[fill=white] (v1) circle[radius=2pt];
\node[above=3pt,inner sep=0pt,font=\scriptsize] at (v1) {$ 1 $};
\draw[fill=white] (v2) circle[radius=2pt];
\node[above=3pt,inner sep=0pt,font=\scriptsize] at (v2) {$ 2 $};
\draw[fill=black] (v3) circle[radius=2pt];
\draw[fill=black] (v4) circle[radius=2pt];
\end{tikzpicture}\\[-0.6ex]
{\small $F^{\tau_2}_{7}$}
\end{minipage}
\begin{minipage}[b]{2.45cm}\centering
\begin{tikzpicture}[x=0.86cm,y=0.86cm,line width=0.6pt]
\path[use as bounding box] (-1.02,-1.04) rectangle (1.02,1.14);
\coordinate (v0) at (-0.75,0.55);
\coordinate (v1) at (0,0.55);
\coordinate (v2) at (0.75,0.55);
\coordinate (v3) at (-0.42,-0.55);
\coordinate (v4) at (0.42,-0.55);
\draw (v0)--(v3);
\draw (v1)--(v3);
\draw (v1)--(v4);
\draw[fill=white] (v0) circle[radius=2pt];
\node[above=3pt,inner sep=0pt,font=\scriptsize] at (v0) {$ 0 $};
\draw[fill=white] (v1) circle[radius=2pt];
\node[above=3pt,inner sep=0pt,font=\scriptsize] at (v1) {$ 1 $};
\draw[fill=white] (v2) circle[radius=2pt];
\node[above=3pt,inner sep=0pt,font=\scriptsize] at (v2) {$ 2 $};
\draw[fill=black] (v3) circle[radius=2pt];
\draw[fill=black] (v4) circle[radius=2pt];
\end{tikzpicture}\\[-0.6ex]
{\small $F^{\tau_2}_{9}$}
\end{minipage}
\begin{minipage}[b]{2.45cm}\centering
\begin{tikzpicture}[x=0.86cm,y=0.86cm,line width=0.6pt]
\path[use as bounding box] (-1.02,-1.04) rectangle (1.02,1.14);
\coordinate (v0) at (-0.75,0.55);
\coordinate (v1) at (0,0.55);
\coordinate (v2) at (0.75,0.55);
\coordinate (v3) at (-0.42,-0.55);
\coordinate (v4) at (0.42,-0.55);
\draw (v0)--(v3);
\draw (v0)--(v4);
\draw (v2)--(v3);
\draw[fill=white] (v0) circle[radius=2pt];
\node[above=3pt,inner sep=0pt,font=\scriptsize] at (v0) {$ 0 $};
\draw[fill=white] (v1) circle[radius=2pt];
\node[above=3pt,inner sep=0pt,font=\scriptsize] at (v1) {$ 1 $};
\draw[fill=white] (v2) circle[radius=2pt];
\node[above=3pt,inner sep=0pt,font=\scriptsize] at (v2) {$ 2 $};
\draw[fill=black] (v3) circle[radius=2pt];
\draw[fill=black] (v4) circle[radius=2pt];
\end{tikzpicture}\\[-0.6ex]
{\small $F^{\tau_2}_{14}$}
\end{minipage}
\begin{minipage}[b]{2.45cm}\centering
\begin{tikzpicture}[x=0.86cm,y=0.86cm,line width=0.6pt]
\path[use as bounding box] (-1.02,-1.04) rectangle (1.02,1.14);
\coordinate (v0) at (-0.75,0.55);
\coordinate (v1) at (0,0.55);
\coordinate (v2) at (0.75,0.55);
\coordinate (v3) at (-0.42,-0.55);
\coordinate (v4) at (0.42,-0.55);
\draw (v1)--(v3);
\draw (v1)--(v4);
\draw (v2)--(v3);
\draw[fill=white] (v0) circle[radius=2pt];
\node[above=3pt,inner sep=0pt,font=\scriptsize] at (v0) {$ 0 $};
\draw[fill=white] (v1) circle[radius=2pt];
\node[above=3pt,inner sep=0pt,font=\scriptsize] at (v1) {$ 1 $};
\draw[fill=white] (v2) circle[radius=2pt];
\node[above=3pt,inner sep=0pt,font=\scriptsize] at (v2) {$ 2 $};
\draw[fill=black] (v3) circle[radius=2pt];
\draw[fill=black] (v4) circle[radius=2pt];
\end{tikzpicture}\\[-0.6ex]
{\small $F^{\tau_2}_{23}$}
\end{minipage}
\begin{minipage}[b]{2.45cm}\centering
\begin{tikzpicture}[x=0.86cm,y=0.86cm,line width=0.6pt]
\path[use as bounding box] (-1.02,-1.04) rectangle (1.02,1.14);
\coordinate (v0) at (-0.75,0.55);
\coordinate (v1) at (0,0.55);
\coordinate (v2) at (0.75,0.55);
\coordinate (v3) at (-0.42,-0.55);
\coordinate (v4) at (0.42,-0.55);
\draw (v0)--(v3);
\draw (v2)--(v3);
\draw (v2)--(v4);
\draw[fill=white] (v0) circle[radius=2pt];
\node[above=3pt,inner sep=0pt,font=\scriptsize] at (v0) {$ 0 $};
\draw[fill=white] (v1) circle[radius=2pt];
\node[above=3pt,inner sep=0pt,font=\scriptsize] at (v1) {$ 1 $};
\draw[fill=white] (v2) circle[radius=2pt];
\node[above=3pt,inner sep=0pt,font=\scriptsize] at (v2) {$ 2 $};
\draw[fill=black] (v3) circle[radius=2pt];
\draw[fill=black] (v4) circle[radius=2pt];
\end{tikzpicture}\\[-0.6ex]
{\small $F^{\tau_2}_{28}$}
\end{minipage}
\begin{minipage}[b]{2.45cm}\centering
\begin{tikzpicture}[x=0.86cm,y=0.86cm,line width=0.6pt]
\path[use as bounding box] (-1.02,-1.04) rectangle (1.02,1.14);
\coordinate (v0) at (-0.75,0.55);
\coordinate (v1) at (0,0.55);
\coordinate (v2) at (0.75,0.55);
\coordinate (v3) at (-0.42,-0.55);
\coordinate (v4) at (0.42,-0.55);
\draw (v1)--(v3);
\draw (v2)--(v3);
\draw (v2)--(v4);
\draw[fill=white] (v0) circle[radius=2pt];
\node[above=3pt,inner sep=0pt,font=\scriptsize] at (v0) {$ 0 $};
\draw[fill=white] (v1) circle[radius=2pt];
\node[above=3pt,inner sep=0pt,font=\scriptsize] at (v1) {$ 1 $};
\draw[fill=white] (v2) circle[radius=2pt];
\node[above=3pt,inner sep=0pt,font=\scriptsize] at (v2) {$ 2 $};
\draw[fill=black] (v3) circle[radius=2pt];
\draw[fill=black] (v4) circle[radius=2pt];
\end{tikzpicture}\\[-0.6ex]
{\small $F^{\tau_2}_{30}$}
\end{minipage}
\caption{The flags in \eqref{eq:examplesquare}.}\label{fig:square}
\end{figure}

\Needspace{10\baselineskip}
\begin{Lem}\label{lem:certificate}
For the four local weights and the $40$ linear combinations specified above, there are rational numbers $\gamma_H>10^{-10}$, $H\in\mathcal H_7$, such that every admissible graph $G$ of order $N\ge7$ satisfies the exact identity
\begin{equation}\label{eq:flagidentity}
r_0\mathcal D(G)-\mathcal E(G)
=\sum_{u}\av{L_u^2}_G
+\sum_{H\in\mathcal H_7}\gamma_H p(H,G).
\end{equation}
In particular,
\begin{equation}\label{eq:countingbound}
r_0\mathcal D(G)-\mathcal E(G)\ge10^{-10}-O(N^{-1}),
\end{equation}
where the implicit constant is independent of $G$.
\end{Lem}

\begin{proof}
We expand all the flag expressions in the common basis $\mathcal H_7$. Index the $40$ columns by $u$ and define
\[
M_{u,H}=\av{L_u^2}_H,\qquad
d_H=\mathcal D(H),\qquad e_H=\mathcal E(H).
\]
The averaging principle, obtained by first choosing a seven-element set and then ordering it, gives
\[
\left(\av{L_u^2}_G\right)_u=M\left(p(H,G)\right)_H,
\quad
\mathcal D(G)=\sum_H d_Hp(H,G),\quad
\mathcal E(G)=\sum_H e_Hp(H,G).
\]
It is therefore enough to make the remainder positive for each $H\in\mathcal H_7$. Multiplying these coordinatewise inequalities by the nonnegative densities $p(H,G)$ and summing will give the required counting inequality. In matrix form, the coefficient comparison is
\begin{equation}\label{eq:coordinate}
\ones^TM<r_0d-e,
\end{equation}
where $d=(d_H)_H$ and $e=(e_H)_H$ are row vectors and the inequality is coordinatewise. We give the exact calculation underlying this comparison.

Write $a=569$ and $b=10000$, and let $A$ denote the adjacency matrix of the graph $H$ under consideration. For an ordering $(v_0,\ldots,v_6)$, let $i,j$ be the indices of the two flags over a given type. An incorrect root type contributes zero. Summing over all orderings and all columns gives the integer
\begin{align}\label{eq:xi}
\Xi_H={}&b\sum_{\tau}\sum_{u\text{ column of }U_\tau}
\sum_{\text{orderings of type }\tau}u_i u_j\notag\\
&+\sum_{\sigma}\sum_{c\text{ column of }V_\sigma}
\sum_{\text{orderings of type }\sigma}
\left[b A_{v_5v_6}(c_i^2+c_i c_j)-a c_i^2\right].
\end{align}
To relate the second line to \eqref{eq:flagD}--\eqref{eq:flagE}, fix a spectral column $c$ and sum over orderings of its root type. Exchanging $x=v_5$ and $y=v_6$ interchanges $i,j$ and preserves $A_{xy}$. Hence
\[
\sum A_{xy}(c_i^2+c_ic_j)=\frac12\sum A_{xy}(c_i+c_j)^2,
\qquad
\sum c_i^2=\frac12\sum(c_i^2+c_j^2).
\]
Thus the spectral contribution agrees with the symmetric flag expressions. Each product of weights contributes a factor $s^{-2}$, averaging over seven-vertex orderings contributes $1/7!$, and multiplication by $b$ clears the denominator of $r_0=a/b$. Therefore
\begin{equation}\label{eq:xiscale}
\Xi_H=5040bs^2\left(\sum_uM_{u,H}+e_H-r_0d_H\right).
\end{equation}
For example, let $H=7K_1$. Only the types $\tau_1,\tau_2,\sigma_1$ contribute. The rows for their empty flags in $U_{\tau_1},U_{\tau_2},V_{\sigma_1}$ are, respectively,
\[
(-34,-5,58,-1588,-1747),\quad
(0,0,0,0,37,0,0),\quad
(4299,10700).
\]
The first two rows have squared norms $5578298$ and $1369$, and $e_H=0$. Hence this coordinate of \eqref{eq:coordinate} can be checked directly:
\[
\sum_u M_{u,H}+e_H-r_0d_H
=\frac{5578298+1369-\frac{569}{10000}(4299^2+10700^2)}{10^{16}}
=-\frac{19864057169}{10^{20}}<0.
\]
All other coordinates follow from the same finite expansion. Table~\ref{tab:coefficients} reports the exact integer comparison for all $89$ graphs, grouped by their numbers of edges. For each edge count $m$, its last column gives $\min_{e(H)=m}(-\Xi_H)$. Every value in this column is positive, and
\begin{equation}\label{eq:certificatemax}
\max_{H\in\mathcal H_7}\Xi_H=-75231473686920.
\end{equation}
It follows that \eqref{eq:coordinate} holds with
\begin{equation}\label{eq:gamma}
\gamma_H=-\frac{\Xi_H}{5040bs^2}
\ge\frac{208976315797}{1400000000000000000000}>10^{-10}.
\end{equation}

\begin{table}[H]
\centering
\caption{Coefficient comparison by edge count.}\label{tab:coefficients}
\begin{tabular}{rrr}
\toprule
Edges $m$ & Number of graphs & Smallest value of $-\Xi_H$\\
\midrule
0 & 1 & 100114848131760\\
1 & 1 & 126589693019520\\
2 & 2 & 98937134956848\\
3 & 4 & 102426248316816\\
4 & 8 & 100265675008752\\
5 & 13 & 75231473686920\\
6 & 19 & 89781404646128\\
7 & 15 & 95843111725544\\
8 & 13 & 86033949817024\\
9 & 7 & 107768130480000\\
10 & 4 & 117695761922616\\
11 & 1 & 1043881219638000\\
12 & 1 & 1846009088942400\\
\bottomrule
\end{tabular}
\end{table}

The graph list can be obtained by enumerating all bipartitions of seven vertices and all subsets of their crossing edges, taking isomorphism classes, and adding $C_7$. The bipartite enumeration contains $103237$ labelled graphs and $88$ isomorphism classes. Each coefficient is then computed from all $5040$ orderings using \eqref{eq:xi}; the expected coefficients are used only as a separate check. This proves \eqref{eq:flagidentity}. Applying Lemma~\ref{lem:flagsquare} to the $40$ squares and using $\sum_Hp(H,G)=1$ proves \eqref{eq:countingbound}.
\end{proof}

Thus flag algebras supply the counting inequality needed after the spectral reduction: the gap $r_0\mathcal D-\mathcal E$ is a sum of averaged squares and an explicitly positive combination of graph densities. Numerical optimisation is used to find the rational coefficients; the displayed identity and the integer calculation establish their validity.


\begin{proof}[Proof of Theorem~\ref{th:main}]
Suppose that an admissible graph $G$ of order $n$ satisfies $\qmin(G)\ge r_0n$. Let $G^{(t)}$ be its balanced complete blow-up, in which each vertex is replaced by an independent set of size $t$ and each edge by a complete bipartite graph. The graph $G^{(t)}$ is still admissible: an odd cycle of length at most five in it would project to an odd closed walk of that length in $G$, and hence to a forbidden odd cycle.

With vertices ordered by their replacement sets,
\[
Q(G^{(t)})=tD(G)\otimes I_t+A(G)\otimes J_t,
\]
where $I_t$ is the identity matrix and $J_t$ is the all-ones matrix. On vectors constant on each replacement set, this matrix acts as $tQ(G)$. On the subspace supported on the replacement set of $v$ with coordinate sum zero, it acts as multiplication by $td(v)$. These invariant subspaces give all its eigenvalues. Since $\qmin(G)\le d(v)$ for every $v$, we obtain
\begin{equation}\label{eq:blowup}
\qmin(G^{(t)})=t\qmin(G)\ge r_0tn.
\end{equation}
Write $N=tn$. Apply \eqref{eq:spectralreduction} to $G^{(t)}$. Because $\mathcal E\ge0$,
\[
r_0\mathcal D(G^{(t)})-\mathcal E(G^{(t)})
\le\frac{5\mathcal D(G^{(t)})-6\mathcal E(G^{(t)})}{N}
\le\frac{5\mathcal D(G^{(t)})}{N}=O(N^{-1}).
\]
The final bound holds because the four weight vectors are fixed. In contrast, Lemma~\ref{lem:certificate} gives
\[
r_0\mathcal D(G^{(t)})-\mathcal E(G^{(t)})
\ge10^{-10}-O(N^{-1}).
\]
Letting $t$ tend to infinity gives a contradiction. This proves the strict finite-graph inequality in \eqref{eq:main}.
\end{proof}

\section{Concluding remarks}\label{sec:conclusion}

For odd $g\ge7$, define
\[
\beta_g=\sup\left\{\frac{\qmin(G)}{|V(G)|}: |V(G)|\ge1,\ G\text{ has odd girth at least }g\right\}.
\]
Theorem~\ref{th:main} gives $\beta_g\le569/10000$ for every such $g$. In fact, the uniform positive remainder in the certificate gives a strict bound for the supremum as well. Set $\gamma_* = \min_H\gamma_H>0$ and $B=\sum_w\max_iw_i^2>0$, where the sum is over the four local weight vectors. Then $\mathcal D(G)\le B$. For a fixed admissible graph $G$, put $r=\qmin(G)/|V(G)|<r_0$. Applying the spectral reduction and the certificate to its balanced blow-ups gives
\[
(r_0-r)\mathcal D(G^{(t)})\ge\gamma_*-O((t|V(G)|)^{-1}).
\]
Since $\mathcal D(G^{(t)})\le B$, letting $t\to\infty$ yields $r\le r_0-\gamma_*/B$. Consequently $\beta_g\le r_0-\gamma_*/B<569/10000$ for every odd $g\ge7$. We retain the simpler coefficient $569/10000$ in the statement of the main theorem.

The folded $7$-cube gives the lower bound $\beta_7\ge1/32$; it was already used for the adjacency spectral ratio by Abiad, Taranchuk and van Veluw \cite{atv}. One model has vertex set $\FF_2^6$ and connection set consisting of the six coordinate vectors and the all-ones vector. No nonempty odd subset of fewer than seven of these generators sums to zero, and the sum of all seven is zero. Thus the graph has odd girth seven. For $a\in\FF_2^6$, the character $\chi_a(x)=(-1)^{a\cdot x}$ is an adjacency eigenvector, with eigenvalue equal to the sum of its values on the seven generators. If $a$ has weight $j$, this sum is
\[
\sum_{i=1}^6(-1)^{a_i}+(-1)^{a_1+\cdots+a_6}
=6-2j+(-1)^j.
\]
The minimum over $0\le j\le6$ is $-5$. The graph is $7$-regular on $64$ vertices, so its least signless Laplacian eigenvalue is $7-5=2$. Equation~\eqref{eq:blowup} gives arbitrarily large examples with the same normalised value. Hence
\begin{equation}\label{eq:gap}
\frac1{32}\le\beta_7<\frac{569}{10000}.
\end{equation}
Determining $\beta_7$ remains open; no optimality claim is made for the upper coefficient in \eqref{eq:gap}.

The local construction in Lemma~\ref{lem:local} extends to any fixed number of roots. With $k$ roots, the same neighbourhood partition gives spectral constraints on $(k+2)$-vertex densities. Combining these with larger flag Gram matrices provides a systematic way to seek stronger bounds. Additional forbidden odd cycles can be imposed by restricting the admissible graphs. Whether these larger systems yield substantially stronger bounds, and whether their certificates admit shorter descriptions, are natural questions.

The complete integer coefficients of the $40$ flag squares and four local test vectors, together with an independent verifier, are supplied in the single SageMath ancillary file \texttt{verify\_certificates.sage} accompanying the arXiv version of this manuscript. The file specifies all labels and averages, enumerates the admissible graphs and flags, and verifies the coefficient comparison \eqref{eq:coordinate} by recomputing \eqref{eq:xi} in exact integer arithmetic. It requires SageMath and the Python standard library, and uses no numerical optimiser, external data files or network connection.

\end{document}